\documentclass[12pt]{amsart}
\usepackage{amsmath,amsthm,amsfonts,amssymb,verbatim}
\usepackage{hyperref}
\usepackage[usenames]{color} 

\usepackage{marginnote}

\newtheorem{thm}{Theorem}[section]
\newtheorem{lemma}[thm]{Lemma}
\newtheorem{cor}[thm]{Corollary}

\newtheorem{prop}[thm]{Proposition}
\newtheorem{question}[thm]{Question}

\usepackage{pdfsync}
\usepackage{graphicx}

\theoremstyle{definition}  

\numberwithin{equation}{section}
\newtheorem{defn}[thm]{Definition}

\newtheorem{exs}[thm]{Examples}
\newtheorem{remark}[thm]{Remark}

\theoremstyle{definition}
\theoremstyle{remark}

\newcounter{enumitemp}

\DeclareMathOperator{\Bo}{O}

\DeclareMathOperator{\Homeo}{Homeo}

\DeclareMathOperator{\lo}{o}

\DeclareMathOperator{\rr}{range}

\DeclareMathOperator{\rank}{rank}

\DeclareMathOperator{\BS}{{\rm BS}}

\newcommand{\C}{{\mathcal C}}

\newcommand{\U}{{\mathcal U}}

\newcommand{\N}{{\mathbb N}}

\newcommand{\cR}{{\mathcal R}}

\newcommand{\Z}{\mathbb Z}

\newcommand{\shift}{subshift }
\newcommand{\shifts}{subshifts }

\newcommand{\shiftno}{subshift}

\def\CL{{\mathcal L}}
\def\CS{{\mathcal S}}

\def\Z{{\mathbb Z}}
\newcommand{\trycomment}[1]{}

\DeclareMathOperator{\Aut}{{\rm Aut}}

\DeclareMathOperator{\topo}{{\rm top}}

\title{Distortion and the automorphism group of a shift}
\author{Van Cyr}
\address{Bucknell University, Lewisburg, PA 17837 USA}
\email{van.cyr@bucknell.edu}
\author{John Franks}
\address{Northwestern University, Evanston, IL 60208 USA}
\email{j-franks@northwestern.edu}
\author{Bryna Kra}
\address{Northwestern University, Evanston, IL 60208 USA}
\email{kra@math.northwestern.edu}
\author{Samuel Petite}
\address{Laboratoire Ami\'enois
de Math\'ematiques Fondamentales et Appliqu\'ees, CNRS-UMR 7352, Universit\'{e} de Picardie Jules Verne, 33 rue Saint Leu, 80039   Amiens cedex 1,
France.} \email{samuel.petite@u-picardie.fr}

\subjclass[2010]{}

\thanks{The third author was partially supported by NSF grant 1500670.}
\dedicatory{Dedicated to the memory of Roy Adler.}

\begin{document}
\maketitle
 \begin{abstract} 
The set of automorphisms of a one-dimensional \shift $(X, \sigma)$ forms a countable, but often very complicated, group.  For zero entropy shifts, it has recently been shown that the automorphism group is more tame. We provide the first examples of countable groups that cannot embed into the automorphism group of any zero entropy \shiftno.  In particular, we show that the Baumslag-Solitar groups $\BS(1,n)$ and all other groups that contain exponentially distorted elements cannot embed into $\Aut(X)$ when $h_{\topo}(X)=0$.  We further show that distortion in nilpotent groups gives a nontrivial obstruction to embedding such a group in any low complexity shift. 
 \end{abstract}

\section{Introduction}
If $\Sigma$ is a finite alphabet and $X\subset\Sigma^{\Z}$ is a closed set that is invariant under the left shift $\sigma\colon\Sigma^{\Z}\to\Sigma^{\Z}$, then $(X, \sigma)$ is called a \shift.  
The collection of homeomorphisms $\phi\colon X\to X$ that commute with $\sigma$ 
forms a group (under composition) called the automorphism group $\Aut(X)$ of the shift $(X,\sigma)$.  
This group is always countable, but for many classical \shifts (including all mixing shifts of finite type) it has a  complicated subgroup structure, containing isomorphic copies of all locally finite, residually finite groups, the fundamental group of any $2$-manifold, the free group on two generators, and many other groups (see~\cite{H, BLR, KR1}). 
On the other hand, for shifts with low complexity (see Section~\ref{sec:background} for precise definitions), there are numerous restrictions that arise (see~\cite{CK1, CK2, DDMP}).  Theorems of this nature typically take the following form: suppose $(X,\sigma)$ is a \shift with some dynamical assumption (such as minimality or transitivity) and suppose that
the complexity function of $(X,\sigma)$ grows more slowly than some explicitly chosen subexponential rate, then $\Aut(X)$ has some particular algebraic property.  Without these growth-rate and dynamical assumptions, little is known about the algebraic structure of $\Aut(X)$ and it was asked in~\cite[Question 6.1]{DDMP} whether every countable group arises as the automorphism group of some minimal, zero entropy shift. 
We answer this question negatively, giving explicit countable groups that cannot embed.  Moreover, we give an algebraic constraint on $\Aut(X)$ that applies to any \shift with zero topological entropy (with no need for further assumptions on the dynamics). 

To show how these constraints arise, we study the types of distortion that can arise (or not) in $\Aut(X)$.  
For a finitely generated group $G$, an element $g\in G$ 
is distorted with respect to a symmetric generating set $\mathcal S$ if the distance of its iterates to the identity grows sublinearly (with 
respect to iteration) in the word-length metric.  A priori, this definition 
depends on the choice of a symmetric generating set $\mathcal S$, but it is well-known that $g$ is distorted with respect to one symmetric generating set if and only if it is with respect to every symmetric generating set. 
Thus we can refer to an element $g$ as distorted without making explicit reference to the set $\mathcal{S}$.  
Distortion can be quantified, depending on how slowly 
the distance of the iterates to the identity grows: we say that $g$ is {\em polynomially distorted} 
if $d_{\CS}(e,g^n)=O(n^{1/d})$ (for some $d\in\N$) and that $g$ is {\em exponentially distorted} if $d_{\CS}(e,g^n)=O(\log(n))$.  

A different notion of distortion is in terms of the range of 
the automorphism: we say that $\phi\in\Aut(X)$ is {\em range distorted} if the size of the shortest block-code defining 
$\phi^n$ grows sublinearly in $n$.  This idea is explored in~\cite{CFK}, where it is shown that if $\phi$ is of infinite order and range distorted then the topological entropy of $(X,\phi)$, rather than that of shift $(X,\sigma)$,  is zero.  
The two notions of distortion are related: if $G\subset\Aut(X)$ is a finitely generated subgroup and if $\phi\in G$ is distorted in $G$, then the automorphism $\phi$ is range distorted (see Proposition~\ref{prop: distortion}). 

One of our main tools is 
the interplay between the level of distortion in $\Aut(X)$ and the growth rate of the complexity function of $(X,\sigma)$.  We use this to study the algebraic structure of  the group $\Aut(X)$. 

In~\cite{CK2}, it is shown that for a minimal shift whose complexity grows 
at most polynomially, any finitely generated, torsion free subgroup of $\Aut(X,\sigma)$ is virtually 
nilpotent.   For very low complexity systems, we improve on this result, showing that for
any shift whose complexity function is $o(n^{((d+1)(d+2)/2)+2})$, any finitely generated, torsion 
free subgroup of the automorphism group is virtually $d$-step nilpotent (the precise statement is in Theorem~\ref{th:embeddingHeisenberg}).  
In particular, if the complexity is $o(n^5)$, 
any finitely generated, torsion free subgroup of $\Aut(X,\sigma)$ is virtually abelian.  
In Theorem~\ref{thm:nilpotent}, we show that the  growth-rate of the complexity of $(X,\sigma)$ 
provides further obstructions for an infinite nilpotent group $G$ to embed into the automorphism group of a shift. 

We would be remiss were we not to acknowledge that examples of
non-abelian lattice actions as shift automorphisms are sorely
lacking; in the setting of low complexity (zero entropy), we have none and  
for positive entropy shifts, we can not rule out some of the simplest non-abelian groups. It is conceivable that few or no such actions exist.  Even if this 
turns out to be the case, it is our hope that some of our results furnish the first steps towards non-existence proofs.

We conclude with several open questions, primarily on what sorts of restrictions 
can be placed on the automorphism group   of a shift.  

\section{Background on shifts}
\label{sec:background}

\subsection{One-dimensional subshifts and automorphisms}
\label{subsec:autos}
We assume throughout that $\Sigma$ is a finite set endowed 
with the discrete topology, and $\Sigma^\Z$ is endowed with the product
topology. For $x \in \Sigma^\Z$, we write $x[n] \in \Sigma$ for the value of $x$ at $n \in \Z$.

The {\em left shift} $\sigma\colon \Sigma^\Z \to \Sigma^\Z$ is defined by 
$(\sigma x)[n] = x[n+1]$, and is a homeomorphism from $\Sigma^\Z$ to itself.
The pair $(X, \sigma)$ is a {\em \shiftno}, or just a {\em shift} when the context is clear, if 
$X\subset \Sigma^\Z$ is a closed set that is invariant under the
left shift $\sigma\colon\Sigma^\Z\to\Sigma^\Z$.  

The system $(X, \sigma)$ is said to be {\em minimal} if the orbit closure of any $x\in X$ is all of $X$.  

An {\em automorphism} of the shift
$(X, \sigma)$ is a homeomorphism $\phi\colon X\to X$ such that
$\phi\circ\sigma = \sigma\circ\phi$.  The group of all
automorphisms of $(X, \sigma)$ is denoted $\Aut(X, \sigma)$, or simply $\Aut(X)$ when $\sigma$ 
is clear from the context.

A map $\phi\colon X\to X$ is a {\em sliding block code} if there exists $R\in\N$ such that 
for any $x,y\in X$ with $x[i] = y[i]$ for $-R \le i \le R$, we have that $\phi(x)[0] = \phi(y)[0]$.
The least $R$ such that this holds is called the {\em range} of $\phi$. 

By the Curtis-Hedlund-Lyndon Theorem~\cite{H}, any automorphism $\phi\colon X\to X$ 
of a shift $(X,\sigma)$ is a sliding block code.
In particular, $\Aut(X)$ is always countable. 

\subsection{The language and complexity of a one-dimensional  subshift}
 The {\em words $\CL_k(X)$ of length $k$  in $X$} are defined to 
be the collection of all $a_1, \dots, a_k \in \Sigma^{k}$ such that
there exist $x \in X$ and $m\in \Z$ with $x[m+i] = a_i$ for $1 \le i \le k$.
The length of a word $w \in \CL(X)$ is denoted by $|w|$.
The language $\CL(X)=\bigcup_{k=1}^\infty\CL_k(X)$ is defined 
to be the collection of all finite words. 

A word  $w \in \CL(X)$  is said to be {\em right  special} (respectively, {\em left special}) if  it can be  extended in the language  in at least two distinct ways to the right (respectively, to the left).  
Thus $w$ is right special if  $\vert \{ x \in \Sigma\colon 
wx \in \CL(X)\} \vert \ge 2$  and $w$ is left special if 
$\vert \{ x \in \Sigma\colon  xw \in \CL(X)\} \vert \ge 2$.  A well-known consequence of the work of  Morse and Hedlund~\cite{MH1}  is that every infinite shift admits a right special word of length $n$ for every $n\geq 1$ 
(similarly for left special words).

The {\em complexity} $P_X\colon \N\to\N$ of the shift $(X, \sigma)$ counts the number of 
words of length $n$ in the language of $X$.  Thus 
$$
P_X(n) = \big \vert\CL_n(X) \big \vert.
$$
The exponential growth rate of the complexity is
the {\em topological entropy $h_{\topo}$} of the shift $\sigma$. Thus
\[
h_{\topo}(\sigma) = \lim_{n \to \infty} \frac{\log(P_X(n))}{n}.
\]
This is equivalent to the usual definition of topological entropy using $(n,\varepsilon)$-separated sets (see, for example~\cite{LM}). 

\subsection{Two-dimensional \shifts} 
With minor modifications, the previous notions may be extended to
higher dimensions. For our needs dimension two suffices, and so we
specialize to that case. The set of
functions $\eta \colon \Z^2 \to \Sigma$ is $\Sigma^{\Z^2}$ endowed
with the product topology. The {\em shift action} of $\Z^2$ on
$\Sigma^{\Z^2}$ is given by $(\sigma^z (\eta))(\cdot) := \eta(\cdot -z)$ for
every $\eta \in \Sigma^{\Z^2}$, $z\in \Z^2$. Every $\sigma^z$ is a
homeomorphism on $\Sigma^{\Z^2}$. A {\em two-dimensional shift} is a
closed subset $X \subset \Sigma^{\Z^2}$ invariant by the shift
action. To avoid confusion with the one-dimensional case we denote by
$(X, \sigma|_{X}, \Z^2)$ the associated dynamical system.

A function $\eta \in \Sigma^{\Z^2}$ is said {\em vertically (resp. horizontally) periodic}  if it is a periodic point for $\sigma^{(0,1)}$ (resp. $\sigma^{(1,0)}$).
We say that a subset $A \subset \Z^2$ {\em codes} a subset $B \subset \Z^2$ if for any $\eta, \theta \in X \subset \Sigma^{\Z^2}$ coinciding on the  set $A$ (in other words,  $\eta|_{A} = \theta|_{A}$)  it follows that $\eta$ and $\theta$ coincide on $B$ (meaning that $\eta|_{B} = \theta|_{B}$).

We give a definition of the   {\em complexity function}  $P_X\colon \{\text{finite subsets of $\Z^2$}\}\to\N$, which is analogous to that for one-dimensional shifts.
Namely, for each finite set $\CS \subset \Z^2$  the value $P_X(\CS)$ is  defined to be  the 
number of legal $X$ colorings of the finite set $\CS\subset\Z^2$.

If $\overline{\mathcal O}_{\Z^2}(\eta)$ denotes the closure of the shift orbit of $\eta \in \Sigma^{\Z^2}$,  
$P_{\overline{\mathcal O}_{\Z^2}(\eta)}(n,k)$ (simply denoted $P_{\eta} (n,k)$)  is the number of distinct 
colorings of  $n\times k$ rectangles $ R(n,k) =\{0, \cdots, n-1\} \times \{0, \cdots, k-1\} \subset \Z^{2}$ which occur in $\overline{\mathcal O}_{\Z^2}(\eta)$, or equivalently the
number of distinct $\eta$-colorings among the sets $R(n,k)+z$ for
$z \in \Z^2$.

\section{Subgroups of the automorphism group}
\label{sec:distortion}

\subsection{Group distortion}
\label{subsec:group-distortion}

\begin{defn}
If  $G$ is a   countable group, 
an element $g\in G$ with infinite order is {\em (group) distorted} if there exists a finite set $S \subset G$ such that
\[
\lim_{n \to \infty} \frac{\ell_{S}(g^n)}{n} = 0, 
\]
where $\ell_{S}(g)$ denotes the length of the shortest presentation of $g$ by elements of $S$ (meaning the word length metric on the group $\langle S \rangle$ generated by $S$  with respect to the generating set $S$). 
\end{defn}

Note that since $\ell_{S}(\cdot)$ is subadditive, this limit exists by
Fekete's Lemma.  Furthermore, this definition also makes sense in a non-finitely generated group $G$. 
Also observe that any  (positive or negative) power  or root of a distorted element is still distorted.

An example  of a distorted element is provided by the {\em discrete Heisenberg group} $\mathcal{H}$,  defined by
\begin{eqnarray}\label{eq:defHeisenberg}
\mathcal{H} =\langle s,t,u\colon su=us, ts =st, [u,t] =utu^{-1}t^{-1} =s\rangle.
 \end{eqnarray}

One can check that for any $n\in\Z$, 
we have that 
$$s^{n^2} = [u^n,t^n] = u^{n}t^{n}u^{-n}t^{-n}$$
and so $s$ is a distorted element of infinite order.

In a similar way, for an automorphism $\phi$ the following limit, called the {\em asymptotic range}, 
$$\rr_{\infty} (\phi) := \lim_{n\to\infty} \frac{\rr(\phi^n)}n = \inf_n\frac{\rr(\phi^n)}n ,$$
exists
(note that the sequence of ranges  $(\rr(\phi^n))_{n\in\N}$ is subadditive). This can be interpreted as the average increase of the range  along powers of $\phi$.  For instance, for the shift map $\sigma$ on an infinite  shift $X$, we trivially  have that for $n \ge 1$, $\rr(\sigma^n) \le n$. 
Since there always exists a right special word of every length, and so in particular of length $2n-1$, 
it follows that   $\rr(\sigma^n) = n$ and so $\rr_{\infty} (\sigma) =1$.

\begin{defn}
An element of $\Aut(X)$ is {\em range distorted} if its asymptotic range is $0$. 
\end{defn}
It follows immediately from the definition that any  power or root of a range distorted automorphism is still range distorted. 

\begin{defn}
For a finite set $S \subset  \Aut(X)$, the {\em range $\cR_{S}$} of the generating set $S$ is defined to be 
\[
\cR_{S} =  \max_{g \in S} \rr(g).
\]
\end{defn}

We check that 
if $g\in G$ is group distorted, then
it is also range distorted: 
\begin{prop}\label{prop: distortion}
If  $G$ is a finitely generated subgroup of $\Aut(X)$ and $g \in G$ is distorted, 
then $g$ is also range distorted and its topological entropy $h_{\topo}(g) = 0$.
\end{prop}
\begin{proof}
Let $S$ denote a symmetric generating set for $G$.  
For all $g_1, g_2\in \Aut(X)$, the range satisfies $\rr(g_1g_2) \le \rr(g_1) + \rr(g_2)$, and so 
it follows that for all $m \in \N$
\[
\rr(g^m) \le \ell_{S}(g^m) \cR_{S}, 
\]
since the element $g$ is
group distorted in $G$.
Moreover, from the fact that $g$ is range
distorted it is not difficult to show that
$h_{\topo}(g) = 0$. This is done, for example in 
 Theorem 5.13 of ~\cite{CFK}.
\end{proof}
However, we do not know if the converse holds, namely if
a range distorted  element of $\Aut(X)$ is a distortion element in the
group $\Aut(X)$.

A consequence of Proposition~\ref{prop: distortion} is that for infinite $X$, 
 the shift map $\sigma$ is never distorted in $\Aut(X)$.  
Of independent interest, since the center of the Heisenberg group is $\langle s \rangle$,  we have: 

\begin{cor}\label{cor:subaction}
Let  $T\colon \mathcal{H} \to \Homeo (Z)$ be a homomorphism from
the Heisenberg group to the
group of self homeomorphisms of a zero-dimensional, compact
metric space $Z$.  Then the subaction $(Z,T({s}))$ is expansive only if $Z$ is finite.
\end{cor}

\begin{proof}
Assume that $(Z,T({s}))$ is expansive. Then it is conjugate to a \shift 
$(X, \sigma)$ by~\cite{KR}. Since $s$ lies in the center of $\mathcal{H}$, the conjugacy maps every element of $T(\mathcal{H})$ into $\Aut(X)$. Since  $T(s)$ is a distorted element, we have that $\sigma$ is range distorted and hence $X$ is finite.
\end{proof}

\begin{defn}\label{def:log}
For a finitely generated group $G$ with generating
set $S$, the element $g \in G$ has {\em exponential distortion} if it has infinite order and
there exists $C > 0$ such that for all sufficiently large $m$, 
\[
\ell_{S}(g^m) \le C\log(m), 
\]
where $\ell_{S}(\cdot)$ denotes the word length of the element in the generating set $S$. 

The smallest such $C$ satisfying this inequality (with the fixed 
generating set $S$) is denoted $\C_{S}$.
\end{defn}

Note that the property of an element having exponential distortion is
independent of the generating set $S$, depending only on the algebraic properties of the group.  
However,  the constant $C$ depends on the
choice of  generators $S$.  

We also say   an element $g$ of infinite order  has {\em polynomial distortion} whenever  $\ell_{S}(g^n) = \Bo(n^{1/d})$ for some finite set $S \subset G$ and integer $d\ge 1$.
Similarly, an automorphism $\phi$ is {\em exponentially} (respectively {\em polynomially}) range  distorted if $\rr(\phi^n) =\Bo(\log n)$ (respectively $\Bo(n^{1/d})$).   

\begin{exs}
The following groups have elements  with exponential distortion :
\begin{itemize}
\item ${\rm SL}(k,\Z)$ for any $k\geq 3$  (see~\cite{LMR}).
\item ${\rm SL}(2, \Z[1/p])$, for any prime $p$ (see~\cite{LMR}).
\item The Baumslag-Solitar group $\BS(1,n) = \langle a,b \colon b a b^{-1} = a^n\rangle$.
\end{itemize}
\end{exs}

To see this for the Baumslag-Solitar group
$\BS(1,n) = \langle a,b \colon b a b^{-1} = a^n\rangle$ with   $n>1$, take the generators $S =\{a, b, b^{-1}\}$.  
Then for any integer $m \ge 2$, write $m$ in base $n$: $m= \alpha_0+\alpha_1 n+\cdots+\alpha_k n^k$ where each $\alpha_{i} \in \{0, \ldots, n-1\}$. 
Using the H\"orner's method, $m=n\cdot(n\cdot(n\cdots(\alpha_{k-1}+n\alpha_k)+\alpha_{k-2})+\cdots+\alpha_1)+\alpha_0$, which implies 
$ a^m= b^ka^{\alpha_{k}}b^{-1}a^{\alpha_{k-1}}b^{-1} \cdots b^{-1}a^{\alpha_{0}}$ and $\ell_{S}(a^m) \le k+n(k+1) +k = \Bo (\log m).$

We show (Corollary~\ref{cor}) that $\BS(1,n)$ does
not embed in the automorphism group of any shift of zero entropy. 

An example of
Hochman~\cite{hochman} gives  a subshift of  polynomial complexity with  an automorphism of infinite order that is  (polynomially) range
distorted but the full automorphism group of the shift constructed is
not explicit and so it is unknown (to us) if this
automorphism is group distorted.

\subsection{Entropy obstructions to embedding}

For a subgroup $G$ of $\Aut(X)$ containing an element $\phi$ with exponential
distortion, the two quantities $\cR_{S}$ and $\C_{S}(\phi)$ determine a 
lower bound on the possible entropy of the shift $(X, \sigma).$ 

\begin{thm}\label{range}
Let  $(X,\sigma)$ be a subshift 
and $\phi\in \Aut(X)$ an element  of infinite order such that for some constant $\cR >0$, $\rr(\phi^m) \le \cR \log(m)$ for each $m \ge 1$.
Then 
\[
h_{\topo}(\sigma) \ge \frac{1}{2\cR}.
\]
\end{thm}

 \begin{proof}
Consider the $\phi$-spacetime $\U$ (see the definition in Section~\ref{subsec:autos})
and let  $V$ be the vertical segment $\{0\} \times \{0, \cdots, 2^n-1\} \subset \Z^2$ of length $2^n$ for $n \ge 1$.

We claim that $V$ is coded by a horizontal segment. 
To prove this, recall that  by definition of the range, if   $x,y \in X$ satisfy  $x[i] = y[i]$ whenever  $|i|  \le \rr(\phi^k)$,  then $\phi^k(x)[0] = \phi^k(y)[0]$. 
So, if $r(n) =\sup_{0\le m \le 2^n} \rr(\phi^m)$,  the horizontal segment 
$H = \{-r(n), \cdots, r(n)\} \times \{0\} $ of length $2r(n) +1$   codes the vertical segment $V$.  Recalling the definition of coding, this means 
that if $\eta, \theta \in\U$ and $\eta |_{H}=\theta |_{H}$, then $\eta |_V=\theta |_V$.

Since $\rr(\phi^m)  \le \cR  \log(m)$ for all $m>0$, it follows that  if  $1 \le m \le 2^n$ and 
$\C_0 = \cR  \log(2)$,  we have that 
\[
\rr(\phi^m)  \le \cR  \log(m) \le \cR  \log(2^n) 
= \C_0  n.
 \]
Thus the horizontal segment of length $2\lfloor \C_{0}n \rfloor +1$ centered at $(0,0)$
codes the vertical segment $V$. We deduce that the number of distinct
vertical words of height $2^n$ that occur in $\U$ is at most
$P_X(2\lfloor \C_{0}n \rfloor+1)$.

Suppose for contradiction that $P_X(2\lfloor \C_{0}n \rfloor+1) \le 2^n$ for some $n\in\N$.  
Then, by  the Morse-Hedlund Theorem~\cite{MH} each  vertical column is  periodic with period at most $2^n$.  
This in turn implies  that $\phi$ has finite order, a contradiction of the hypothesis.  

Thus we have $P_X(2\lfloor \C_{0}n \rfloor+1) > 2^n$ and hence
\[
h_{\topo}(\sigma) = \lim_{n \to \infty} \frac{\log(P_X(n))}{n}
= \lim_{n \to \infty} \frac{\log(P_X(2\lfloor \C_{0}n \rfloor+1))}{2\lfloor \C_{0}n \rfloor+1}
\ge \lim_{n \to \infty} \frac{\log(2^n)}{2\C_0n+1} = \frac{\log(2)}{2\C_0}.
\]
Since $\C_0 = \cR \log(2)$, we conclude that
\[
h_{\topo}(\sigma) \ge \frac{1}{2\cR}.
\qedhere
\]
\end{proof}

\begin{remark}\label{log remark}
Recall that for a finite set $S$ of  generators for a
subgroup $G \subset  \Aut(X)$, the range, $\cR_{S}$, of  $S$ is
defined to be $\cR_{S} =  \max_{g \in S} \rr(g)$.
Also we defined $\C_S(\phi)$ to be the smallest $C$ such that $\ell_{S}(\phi^m) \le C\log(m),$ 
for all $m>0.$

For all $g_1, g_2\in \Aut(X)$, the range satisfies $\rr(g_1g_2) \le \rr(g_1) + \rr(g_2)$.
It follows that for all $m \in \N$,
\[
\rr(g^m) \le \ell_{S}(g^m) \cR_{S} \le \cR_{S} \C_S(\phi)\log(m).
\]
Hence the number $\cR := \cR_S \C_S(\phi)$ satisfies the 
hypothesis of Theorem~\ref{range} and we conclude that
\[
h_{\topo}(\sigma) \ge \frac{1}{2 \cR_{S} \C_S(\phi)}.
\]

The quantity $\C_{S}(\phi)$ depends only on the algebraic properties of the
abstract group $G$ and not on the realizations of these 
automorphisms as sliding block codes, whereas 
$\cR_{S}$ depends only on the range of the sliding block code generators of 
$S$.
\end{remark}

In a private communication, Hochman indicated how to modify the  construction in~\cite{hochman} to obtain  an infinite order, exponentially range distorted automorphism.

Recall that a group $G$ is {\em almost simple} if every normal subgroup
is either finite or has finite index.  The Margulis normal
subgroups theorem (see~\cite{Mar}) implies that many Lie group lattices
are almost simple (including for example ${\rm SL}(n,\Z)$ for $n \ge 3$).

\begin{cor}\label{cor}
Let  $(X,\sigma)$ be a shift with
zero entropy.  Suppose
$G$ is   group and some element $g\in G$ has
exponential distortion.
Then if $\Phi\colon G \to \Aut(X)$ is a homomorphism, the element 
$\Phi(g) \in \Aut(X)$ has finite order.  \\
Moreover, if $G$ is almost simple, then $\Phi(G)$ is a finite group.
\end{cor}

\begin{proof}
If $\phi = \Phi(g)$ is not of finite order, then it is an element
with exponential distortion in the  subgroup $\Phi(G)$ of $\Aut(X)$.  
Moreover, since the range is subadditive, there are  a finite set $S \subset \Phi(G)$  and  positive constants $\cR_{S}$, ${\mathcal C}_{S}$, such that  $\rr(\phi^k) \le \cR_{S} {\mathcal C}_{S} \log(k)$ for each $k \ge 1$. (See Remark~\ref{log remark}.)
This assumption would contradict Theorem~\ref{range}.

Suppose now that $G$ is an almost simple group and
$\Phi\colon G \to \Aut(X)$ is a homomorphism.  If  $g \in G$  has exponential
distortion then, as above, $\phi = \Phi(g)$ has finite order.
So the kernel $K$ of $\Phi$ 
contains infinitely many distinct powers of $g$ and, in particular, $K$ is infinite. 
But since $G$ is almost simple, this implies $K$ has finite index and we conclude
that $\Phi(G) \cong G/K$ is finite.
\end{proof}

Since ${\rm SL}(k,\Z), k \ge 3$ and the Baumslag-Solitar group $\BS(1,n)$ have elements which
are exponentially distorted,
Corollary~\ref{cor} implies
they are examples of finitely generated groups that do not embed into the 
automorphism group of any shift with zero entropy.  In particular, 
this provides an answer Question 6.1 of~\cite{DDMP}.  
However,  we are unable to give a positive entropy shift for which  ${\rm SL}(k,\Z),
k \ge 3$ or $\BS(1,n)$ do not embed.  

On a related note, if
$m>1$ and $n>m$, then $\BS(m,n)$ is not residually
finite~\cite{Mes}. Thus if $X$ is a mixing shift of finite type, then $\BS(m,n)$
does not embed in $\Aut(X)$.

\section{Torsion free nilpotent groups}\label{sec:nilpotent}

\subsection{Periodicity in two dimensions}
We recall some results about two-dimensional shifts which we then use to describe properties 
of the automorphism group of a one-dimensional shift.

\begin{thm}[Cyr \& Kra~\cite{CK}]\label{Cyr-Kra}
Let $\eta\colon\Z^2\to\Sigma$ and suppose there exist $n, k\in\N$ such that $P_{\eta}(n,k) \le \frac{nk}{2}$.  Then there exists $(i,j)\in\Z^2\setminus\{(0,0)\}$ such that $\eta(x+i,y+j)=\eta(x,y)$ for all $(x,y)\in\Z^2$. 
\end{thm} 


\begin{lemma}\label{lem:Cyr-Kra2b}
Let $(X, \sigma, \Z^{2})$ be a two-dimensional subshift such that each element is vertically periodic. 
Then there exists a constant $T>0$ such that each element of $X$ is fixed by $\sigma^{(0,T)}$. 
\end{lemma}

\begin{proof}
Let $Z$ be the collection of  all the sequences along the vertical columns of elements in $X$. The set $Z$ defines a one-dimensional subshift where each sequence is periodic.

If the subshift $Z$ is infinite, its language  contains arbitrarily long right special words. Taking an accumulation point, there exist two different sequences $x,y \in Z$  sharing the same past. This is impossible because  $x$ and $y$ are both periodic. Hence the set $Z$ is finite. So a power $T$ of the shift map is the identity on $Z$. This shows the lemma. 

\end{proof}

\subsection{Complexity obstructions to embedding}
In this section, we show a subshift with an infinite order
polynomially range distorted automorphism cannot have a sub-polynomial
complexity. Then we deduce a restriction on the complexity of a shift
which contains a nilpotent group in its automorphism group.

We start with a sufficient condition for an automorphism to be non-distorted: 
\begin{lemma}\label{lem:range} 
Let $(X,\sigma)$ be a (one-dimensional) shift and let $\phi\in\Aut(X)$.  If there exist $i,j\in\Z\setminus\{0\}$ and an aperiodic $x\in X$ such that $\phi^i(x)=\sigma^j(x)$, then 
$$ 
\rr(\phi^{im})\geq|j|\cdot m 
$$ 
for all $m\in\N$. In particular $\phi$ is not range distorted.
\end{lemma} 
\begin{proof}
Since $\phi$ and $\sigma$ commute, if $\phi^i(x)=\sigma^j(x)$ then $\phi^i(\sigma^kx)=\sigma^i(\sigma^kx)$ for all $k\in\Z$ and so by continuity we get $\phi^{i} = \sigma^j$ on $\overline{\mathcal{O}(x)}$, the orbit closure under $\sigma$ of $x$. The map $\phi^{i}$ preserves $\overline{\mathcal{O}(x)}$ and by aperiodicity, it is infinite. 
For each $m \ge 1$, there is a left special word in $w_m\in\CL( \overline{\mathcal{O}(x)})$ of length $2|j| m-1$, meaning there exist $a\neq b$ such that $aw_m, bw_m\in\CL( \overline{\mathcal{O}(x)})$.  If $y,z\in\overline{\mathcal{O}(x)}$ are such that 
$y[-|j|m+1] \dots y[|j|m-1] = z[-|j|m+1] \dots z[|j|m-1] =w_{m}$ 
but $y[-|j|m]=a$ and $z[-|j|m]=b$ then $(\sigma^{|j|m}y)[0]=a\neq b=(\sigma^{|j|m}z)[0]$.  This implies that $\rr(\sigma^{jm}) \ge  |j| m$, proving the lemma.
\end{proof}

\begin{thm}\label{thm:nilpotent}
Suppose $(X, \sigma)$ is a shift such that there is an automorphism $\phi$ with $\rr(\phi^n) = \Bo(n^{1/d})$.
If $\phi$ has infinite order, then
\[
\liminf_{n \to \infty} \frac{P_X(n)}{n^{d+1}} > 0.
\]
\end{thm}
Recall that~\cite{hochman} provides an example of a subshift with polynomial complexity and an infinite order automorphism polynomially range distorted. Furthermore, the exponent may be arbitrairly small.

\begin{proof}
Let $C_0$ be a constant  such that for any $n \in \N $  and  all integers $k \le n^d$, $\rr(\phi^k) \le C_0n$.

Consider the $\phi$-spacetime $\U$ and let $V$ be a rectangle of
height $n^d$ and width $2n+1$ in $\Z^2$, with the horizontal base of $V$ centered at
$(0,0)$.  
Recall that an horizontal segment of length $2 \rr(\phi^k) +1 $ centered at the origin  codes the point $\{(0,k)\}$.

Let $r(n) = \sup_{ 0\le k \le n^d }2\rr(\phi^k) +2n+1$.  So, the horizontal segment of length $r(n)$  centered at
$(0,0)$   codes  $V$.

 Since  $\rr(\phi^k) \le C_0n$,   we have
that $r(n) \le 2C_0n +2n+1 \le Cn$, where $C = \lfloor 2C_0 +2 \rfloor +2$.
We conclude that there are at most $P_X( Cn)$ possible colorings of 
the rectangle $V$.

Again letting $P_\U(k,n)$ denote the complexity of the $k \times n$ rectangle in $\U$,
this remark  implies that $P_\U(n, n^d) \le P_X(Cn)$. We proceed by 
contradiction and assume that $\liminf_{n} P_X(n) /n^{d+1} =0$.
Since  for each $n$, $P_{X}(C\lfloor n/C \rfloor) \le P_{X}(n)$,  we also have
$\liminf_n P_\U(n, n^d)  /n^{d+1} =0$.  
It follows that $P_\U(n, n^d) < n^{d+1} /2$ for infinitely many $n\in\N$.
By Theorem~\ref{Cyr-Kra}, we conclude that if
$x_0\in X$ is a fixed aperiodic element of $X$, then
$\phi^i(x_0) = \sigma^j(x_0)$ for some $i>0$ and $j \in\Z.$

By Lemma~\ref{lem:range}, 
$\rr(\phi^{im}) \ge |j|\cdot m$ for all $m\in\N$.  On the other hand, since $\phi$ is distorted
we also have that $\lim \rr(\phi^k)/k = 0$.  These two properties can only
be simultaneously true if $j = 0$.  Therefore, for any aperiodic $x_0\in X$,
there exists $i_{x_0}\in\N$ such that $\phi^{i_{x_0}}(x_0)=x_0$. Hence, the map $\phi$ is periodic on each aperiodic sequence of $X$. 

Since the set of periodic sequences of a given period is finite and the automorphism $\phi$ has to preserves this set, the map $\phi$ is also periodic on each periodic sequence. 
By Lemma \ref{lem:Cyr-Kra2b} applied to the $\phi$-spacetime $\U$, the automorphism $\phi$ has a finite order.
But  this 
contradicts the hypothesis that $\phi$ has infinite order.
\end{proof}

Let us recall some basics on nilpotent groups.
If $G$ is a group and $A, B\subset G$, let $[A,B]$ denote the {\em commutator subgroup}, 
meaning the subgroup generated by $\{a^{-1}b^{-1}ab\colon a\in A, b\in B\}$.  
Given a group $G$, we inductively define the lower central series by setting 
$G_1 = G$ and $G_{k+1} = [G, G_k]$ for $k>0$.  
If $d$ is the least integer such that $G_{d+1}$ is the trivial group $\{e\}$, 
then we say that $G$ is {\em $d$-step nilpotent}, and we say that $G$ is {\em nilpotent} if it is $d$-step nilpotent for some $d\geq 1$. 

We use a few standard facts about nilpotent groups: 
\begin{enumerate}
\item  Any subgroup of a finitely generated nilpotent group $G$ is finitely generated. 
\item The set of elements of finite order in a nilpotent group form a normal subgroup $T$, called the {\em torsion subgroup}. \item A finitely generated torsion subgroup of a nilpotent group is finite.
\end{enumerate}

We also use the following standard fact about commutators in any group
(see 2.3b of~\cite{WG} for a more general statement  and further references):

\begin{prop}\label{prop: commutators}
For any group $G$, if $m_i, m_j \ge 1$  and $g_i \in G_i$, and  $g_j \in G_j$ then 
\[
[g_i^{m_i}, g_j^{m_j}] = [g_i, g_j]^{m_i m_j} \mod( G_{i+j+1})
\]
\end{prop}

\begin{lemma}\label{nilpotent-facts}
Suppose $G$ is a finitely generated nilpotent group with torsion
subgroup $T$ and assume that the quotient $G/T$ is $d$-step nilpotent with 
$d\geq 2$.  Then
there exists an element $z \in G_{d}$ of infinite order that is polynomially distorted. 
More precisely, there exists a finite set $S \subset G$ such that
$$ \ell_{S}(z^n) = \Bo(n^{\frac 1d}).$$ 
\end{lemma}

\begin{proof}
We first claim that it suffices to prove the result when $T$ is trivial. Namely, since 
 $T$ is normal and finite, for any $z \in G$,
$$
 \ell_{S}(z^n) \le \ell_{S_0}((zT)^n) + K, 
$$
where $S_0$ is a set of generators for $G/T$,
$K$  is the order of $T$, and $S$ is a set of generators of $G$ containing $T$ and a
representative of each coset in $S_0$.  
Hence it suffices to show that in the torsion free
group $G/T$,  there is an element $zT \in (G/T)_d$ such that
\[
\ell_{S_0}((zT)^n) = \Bo(n^{\frac 1d}).
\] 
Moreover the element $z\in G$ has infinite order as soon as  $zT$ is not $T$.

Thus we now assume that $H = G/T$  is torsion free and $d$-step nilpotent. In particular, 
the group  $H_{d}$ is nontrivial. Since it is generated by the elements $[a_{1}, [a_{2},\dots , [a_{d-1}, a_{d}] \dots ]]$, there exist $a_{1}, \ldots, a_{d}$ in $H$ such that 
 $$z = [a_{1}, [a_{2},\dots , [a_{d-1}, a_{d}] \dots ]]$$ is not trivial.

By Proposition~\ref{prop: commutators},  for any $m_1, \ldots, m_{d} \in\N$,
$$
 [a_1^{m_1},[a_2^{m_2},[ a_3^{m_3},[\dots a_n^{m_d} ]\dots]]] 
= z ^{\Pi_{i=1}^d m_i}.
$$
In particular, for any integers $1 \le q$, $0 \le \alpha \le q$ and $0 \le i <d$,
\[
z^{\alpha q^{i}}=  [a_1^{q},[a_2^{q},[ \dots [a_{i}^q, [a_{i+1}^\alpha, [a_{i+2},[\dots  ,a_d ]\dots]]]\dots]].
\]
Letting $S_{0}$ denote the finite set $\{a_{1}, \ldots, a_{d}\} $, the word length of the right-hand side  of this equation is
\begin{equation}\label{eq:three} 
\ell_{S_{0}}(z^{\alpha q^{i}}) \le 1+\sum_{j=1}^{i} 2^j q+ 2^{i+1}\alpha+  \sum_{j=i+2}^d 2^j   \le 2^{d+1}q. 
\end{equation}

For an integer $n \ge 1$, let $q$ be the smallest integer such that $q> n^{\frac1d}$. Write $n$ in base $q$  as
$$ n = \sum_{i=0}^{d-1} \alpha_{i} q^{i},$$
where $ 0\le \alpha_{i} <q$. Since $\ell_{S_{0}}(ab) \le \ell_{S_{0}}(a) + \ell_{S_{0}}(b)$ for every $a,b \in H$, the inequality in~\eqref{eq:three} leads to 
\begin{equation*}
\ell_{S_{0}}(z^n) \le \sum_{i=0}^{d-1} \ell_{S_{0}}(z^{\alpha_{i} q^{i}}) \le d 2^{d+1}q \le d 2^{d+1}(n^{\frac 1d} +1).
\qedhere
\end{equation*}
\end{proof}

We deduce the following corollary
\begin{cor}\label{cor:nilpotent}
Suppose $(X, \sigma)$ is a shift 
and that $G$ is a finitely generated 
nilpotent subgroup of $\Aut(X)$ with torsion subgroup $T$.
If $G/T$ is $d$-step nilpotent with $d \ge 2,$ then 
\[
\liminf_{n \to \infty} \frac{P_X(n)}{n^{d+1}} > 0.
\]
\end{cor}
 \begin{proof}
 Let $z\in G$ be the element guaranteed to exist by Lemma~\ref{nilpotent-facts}. 
If  $S \subset G$ is a finite set, for all $n\in \N$ we have 
$$\rr(z^n) \le \ell_{S}(z^n)  \max_{g \in S} \rr(g). 
$$
The result follows from Theorem \ref{thm:nilpotent}.
 \end{proof}

\subsection{The automorphism group for \shifts whose complexity is subpolynomial} 

For minimal shifts of polynomial growth, there are strong constraints on the automorphism group: 

\begin{thm}[Cyr \& Kra~\cite{CK2}]
\label{thm:heisenberg} 
Suppose $(X, \sigma)$ is a minimal shift and there exists $\ell\in\N$ such that $ P_{X}(n) =\lo(n^{\ell+1})$.  Then any finitely generated, torsion-free subgroup of $\Aut(X)$ is a group of polynomial growth of degree at most $\ell$. 
\end{thm} 

For instance, if the Heisenberg group is embedded  into the automorphism group of a minimal shift $(X, \sigma)$, 
we must have at least $\limsup_{n} P_{X}(n)/n^4 >0$. 
Using distortion, we obtain a better bound, and we start with  an algebraic lemma on the growth rate of the nilpotent group: 
\begin{lemma}\label{lem:alg2}
If $G$ is a finitely generated, torsion free $d$-step nilpotent group for some $d \ge 2$, then $G$ has polynomial growth rate  of degree at least $d(d+1)/2 +1$.
\end{lemma} 

\begin{proof}
Letting $Z(H)$ denote the center of the group $H$, we inductively define a
sequence of normal subgroups.  Set $Z_{0}(G) =\{1\}$.  Given $Z_{i}(G)$, 
let $\pi_{i} \colon G \to G/Z_{i}(G)$ denote the quotient map and define
$$Z_{i+1}(G) := \pi_{i}^{-1}(Z(G/Z_{i}(G))).$$
By induction on $i$, it is easy to check that $G_{d+1-i}$ is a subgroup of $Z_{i}(G)$. 
By a result of Mal$'$cev~\cite{M1, M2}, each quotient $Z_{i+1}(G)/Z_{i}(G)$ is torsion free. 
Hence,  each group  $G_{d-i}/G_{d-i+1}$ is torsion free, 
as it  embeds into $Z_{i+1}(G)/Z_{i}(G)$.  In particular, the rank of  each $G_{d-i}/G_{d-i+1}$ is at least $1$.

We next check that the rank of $G/G_2$ is at least $2$. 
Let $\bar{G}$ denote the group $G/G_{3}$. It is a  nilpotent group of step at most $2$ 
and the group  $\bar{G}/\bar{G}_{2}$ is abelian.  We claim that $\bar{G}/\bar{G}_{2}$ is not cyclic. If not, 
then $\bar{G}/\bar{G}_{2}$ is generated by the coset $x\bar{G}_{2}$ and so 
$\bar{G}$ is generated by $\bar{G}_{2}$ and $x$. Since the generators commute (recall that $\bar{G}_{2}$ lies in the center of $\bar{G}$), it follows that $\bar{G}$ is abelian.  However, this contradicts the assumption that $G$ 
is $d$-step for some $d\geq 2$.  Therefore,   $G/G_{2}$ has at least two independent generators, and so its rank is 
at least $2$.

By the Bass-Guivarc'h formula~\cite{Bass, G}, $G$ has polynomial growth rate of degree
\begin{equation}
\label{eq:BG}
 \sum_{k\ge 1} k \rank(G_{k}/G_{k+1}),
 \end{equation}
where $\rank(G_{k}/G_{k+1})$ is the torsion free rank of the abelian group $G_{k}/G_{k+1}$. Since
the rank of each  $G_{k}/G_{k+1}$, $1\le k \le d$ is positive and 
$\rank(G_{1}/G_{2}) \ge 2$, the lemma follows.   
\end{proof}

Recall that a group $G$ is {\em virtually nilpotent} (of degree $d$)  if it contains a finite index ($d$-step) nilpotent subgroup.  
 \begin{thm}\label{th:embeddingHeisenberg}
Let  $(X,\sigma)$ be an infinite minimal shift such that for some $d\geq1$ we have $P_{X}(n) =\lo(n^{(d+1)(d+2)/2 +2})$.  Then  any finitely generated, torsion-free subgroup of  $\Aut(X)$ is virtually nilpotent of step at most $d$.
\end{thm}
In particular,  for an aperiodic  minimal shift such that $P_{X}(n) =\lo(n^5)$, any finitely generated, torsion-free subgroup of  $\Aut(X)$ is virtually abelian.
 
 \begin{proof}
Let $G < \Aut(X)$ be a finitely generated, torsion-free subgroup of $\Aut(X)$. 
Theorem \ref{thm:heisenberg} ensures that $G$ has a polynomial growth of degree at most $(d+1)(d+2)/2 +1$. 
By  Gromov's Theorem~\cite{Gromov},  $G$ contains a nilpotent subgroup $H$ with finite index. 
We proceed by contradiction and assume that $H$ is a $k$-step nilpotent group for some $k > d$.

Assume first that   $\langle \sigma \rangle  \cap H  = \{1\}$. Then the group $\Aut(X)$ contains $\langle \sigma \rangle \oplus {H}$, and by Lemma \ref{lem:alg2}, this group has polynomial growth of degree at least  $k(k+1)/2 +2$.  But 
this is a contradiction of Theorem~\ref{thm:heisenberg}.

Otherwise, we assume that $\langle \sigma \rangle  \cap H $ is not trivial. 
Then the group $H /(\langle \sigma \rangle  \cap H)$ is nilpotent. 
Let  $z \in H_{k}$ be the element given by Lemma~\ref{nilpotent-facts}. 
Thus $z$  is distorted  and $z \not\in  \langle \sigma \rangle  \cap H$,  since any element in $ \langle \sigma \rangle $ is not distorted (see the computations in Section~\ref{sec:distortion}). 
It follows that $H /(\langle \sigma \rangle  \cap H)$ is $k'$-step nilpotent for some $k' \ge k$.
By Lemma~\ref{lem:alg2}, this group has polynomial growth of degree at least $k(k+1)/2 +1$. 
 Since  $\langle \sigma \rangle  \cap H $ is an infinite, finitely generated group,  $H$ has polynomial growth rate of degree at least $k(k+1)/2 +2$ (see~\cite[Proposition 2.5 (d)]{Mann} for instance). 
Again, this contradicts Theorem~\ref{thm:heisenberg}.
\end{proof}

In fact one can extract from the proof a more general, but more technical, statement, relating the homogeneous dimension
given by~\eqref{eq:BG} to the step of any finitely generated, torsion-free subgroup of the automorphism group 
for an infinite minimal shift.  

\section{Open questions}\label{sec: questions}

\begin{question}\label{question:1}
Does the discrete Heisenberg group embed into the automorphism 
group of a one-dimensional shift?   More generally, does the automorphism group of a one-dimensional shift have  a distorted element of infinite order?
\end{question}

Interest in the Heisenberg group in particular arises from Theorem~\ref{thm:heisenberg}.  
Consequently, Question~\ref{question:1} becomes most interesting if $X$ is assumed to be minimal and have $P_X(n)=O(n^d)$ as we then have a dichotomy in the possible behaviors.  If there exists a \shift such 
that the Heisenberg group embeds in its automorphism group, 
then Question~\ref{question:1} is resolved affirmatively.  
If no such system $(X, \sigma)$ exists, then by Theorem~\ref{thm:heisenberg} any finitely generated, torsion-free subgroup of $\Aut(X)$ is virtually abelian, 
as the Heisenberg group is a subgroup of any finitely generated, torsion-free, nonabelian nilpotent group, resolving Question~\ref{question:1} negatively.

More generally we have the same question for higher dimensions: 
\begin{question}
Does the discrete Heisenberg group, or more generally a finitely generated group with a distorted element of infinite order, embed 
into the automorphism group of a shift of dimension greater than one?
\end{question}

\begin{question} 
Does a group with exponentially distorted elements, for example ${\rm SL}(3,\Z)$ or the Baumslag-Solitar group $\BS(1,n)$, 
embed into the automorphism group of some positive entropy shift? 
\end{question} 
By Corollary~\ref{cor}, these groups do not embed into $\Aut(X)$ for any shift $X$ with entropy zero. 
We note that if $\BS(1,p)$ embeds in $\Aut(X, \sigma)$ for
some subshift of finite type $\sigma$ and some prime $p$, this
would answer both questions 3.4 and 3.5 of~\cite{BLR} which ask if some
some automorphism of infinite order has an infinite chain of
$p^{th}$ roots.  If $G \cong \BS(1,p)$ is a subgroup of $ \Aut(X)$ and has generators
$a, b$ with relation $b^{-1}a b = a^p$, then it is straightforward to show
that $c_k := b^k a b^{-k}, k \ge 0$ satisfies $c_k^p = c_{k-1}$ and $c_0 = a$, and
so $a$ has an infinite chain of   $p^{th}$ roots.

\end{document}